\documentclass[11pt]{amsart}

\usepackage[margin=1.25in]{geometry}
\usepackage{graphicx}
\usepackage{amsmath, amssymb, amsthm, graphicx, enumerate, tikz, float, color, bbm}
\usepackage{mathtools, comment}
\usepackage{wasysym}
\usepackage[misc]{ifsym}
\usepackage[colorlinks]{hyperref}

\hypersetup{citecolor=blue}
\usetikzlibrary{matrix,arrows,decorations.pathmorphing}

\newtheorem{theorem}{Theorem}[section]

\newtheorem{proposition}[theorem]{Proposition}
\newtheorem{corollary}[theorem]{Corollary}

\theoremstyle{definition}
\newtheorem{definition}[theorem]{Definition}
\newtheorem{example}[theorem]{Example}

\theoremstyle{remark}
\newtheorem{remark}[theorem]{Remark}

\numberwithin{equation}{section}

\DeclareMathOperator{\Hom}{Hom}
\DeclareMathOperator{\Ker}{Ker}
\DeclareMathOperator{\Ig}{Im}
\DeclareMathOperator{\Der}{Der}

\makeatletter
\@namedef{subjclassname@2020}{%
  \textup{2020} Mathematics Subject Classification}
\makeatother

\begin{document}

\allowdisplaybreaks

\title{Secondary Hochschild cohomology and derivations}

\author{Kylie Bennett}
\address{Department of Mathematics, St. Norbert College, De Pere, WI 54115}
\email{kylie.bennett@snc.edu}

\author{Elizabeth Heil}
\address{Department of Mathematics, St. Norbert College, De Pere, WI 54115}
\email{elizabeth.heil@snc.edu}

\author{Jacob Laubacher}
\address{Department of Mathematics, St. Norbert College, De Pere, WI 54115}
\email{jacob.laubacher@snc.edu}


\date{\today}

\keywords{Hochschild cohomology, cyclic cohomology, derivations.\\\indent\emph{Corresponding author.} Jacob Laubacher \Letter~\href{mailto:jacob.laubacher@snc.edu}{jacob.laubacher@snc.edu} \phone~920-403-2961.}

\begin{abstract}
In this paper, we introduce a generalization of derivations. Using these so-called secondary derivations, along with an analogue of Connes' Long Exact Sequence, we are able to provide computations in low dimension for the secondary Hochschild and cyclic cohomologies associated to a commutative triple. We then establish a universal property, which paves the way to relating secondary K\"ahler differentials with the aforementioned secondary derivations.
\end{abstract}

\maketitle

\section{Introduction}

Hochschild cohomology was introduced by Hochschild himself in 1945 in \cite{H}, and one of its generalizations, the aptly named secondary Hochschild cohomology, was brought to light by Staic in 2016 in \cite{S}. In 1964, Gerstenhaber employed the former to study deformations of an algebra $A$ over a field $\mathbbm{k}$ in \cite{G}, whereas Staic used his secondary case to investigate deformations of $A$ with a nontrivial $B$-algebra structure. This structure is induced by a morphism of $\mathbbm{k}$-algebras $\varepsilon:B\longrightarrow A$, and we encode this in the form of a triple $(A,B,\varepsilon)$.

The secondary Hochschild cohomology has many similar properties to that of the usual Hochschild cohomology (see \cite{CSS}, \cite{LSS}, and \cite{SS}, among others). One that we are interested in is with the usual Hochschild cohomology and the set of derivations.

Our goal in this paper is to study how the secondary Hochschild cohomology associated to the triple $(A,B,\varepsilon)$ relates to a generalization of the set of derivations. Section \ref{sec2} provides all the necessary preliminary information, making this paper as self-contained as possible. Next, in Section \ref{sec3}, we introduce the notion of secondary derivations. Consequently, we will then showcase computations in low dimension for both the secondary Hochschild and cyclic cohomologies (see Theorem \ref{HH1ABCo} and Corollary \ref{HC1ABCo}, respectively). Finally, Section \ref{sec4} discusses a universal property for these secondary derivations, and highlights how they can be used to prove that secondary K\"ahler differentials from \cite{L2} are indeed nontrivial.

\section{Preliminaries}\label{sec2}

As is customary, we set $\mathbbm{k}$ to be a field containing $\mathbb{Q}$, and we define $\otimes=\otimes_\mathbbm{k}$ unless otherwise stated. Next we fix all $\mathbbm{k}$-algebras to be necessarily associative with a multiplicative unit. Concerning secondary Hochschild cohomology, it is taken over a triple, which is described as follows:

\begin{definition}(\cite{S})
We call $(A,B,\varepsilon)$ a \textbf{triple} if $A$ is a $\mathbbm{k}$-algebra, $B$ is a commutative $\mathbbm{k}$-algebra, and $\varepsilon:B\longrightarrow A$ is a morphism of $\mathbbm{k}$-algebras such that $\varepsilon(B)\subseteq\mathcal{Z}(A)$. Call $(A,B,\varepsilon)$ a \textbf{commutative triple} if $A$ is also commutative.
\end{definition}

Triples have now been studied quite broadly, and examples, extensions, and applications can be found in a number of places (see \cite{BBN}, \cite{C0}, \cite{CHL}, \cite{CL}, \cite{CLS}, \cite{CSS}, \cite{DMN}, \cite{L1}, or \cite{SS}, for example). When convenient and appropriate, we will denote a triple by $\mathcal{T}=(A,B,\varepsilon)$, so as to make it easier with notation.

\subsection{The secondary Hochschild cohomology}\label{shcsec}

Next we recall the secondary Hochschild cohomology, which was introduced by Staic in \cite{S} in 2016. In that paper, Staic studied the secondary Hochschild cohomology of the triple $(A,B,\varepsilon)$ with coefficients in $M$, which was used to study deformations of $A$ that have a nontrivial $B$-algebra structure. A few years later in \cite{LSS}, the authors established the secondary Hochschild cohomology associated to the triple $(A,B,\varepsilon)$, done through simplicial structures, many details of which can be found in \cite{L0}. This construction is what we will use here.

For notation, we will follow the convention made in \cite{LSS}. For a triple $(A,B,\varepsilon)$, we define $\overline{C}^n(A,B,\varepsilon)=\Hom_\mathbbm{k}(A^{\otimes n+1}\otimes B^{\otimes\frac{n(n+1)}{2}},\mathbbm{k})$. As is customary (see \cite{CS}, for instance), we view the elements of $A^{\otimes n+1}\otimes B^{\otimes\frac{n(n+1)}{2}}$ organized as the matrix
$$
\otimes
\begin{pmatrix}
a_0 & b_{0,1} & b_{0,2} & \cdots & b_{0,n-2} & b_{0,n-1} & b_{0,n}\\
1 & a_1 & b_{1,2} & \cdots & b_{1,n-2} & b_{1,n-1} & b_{1,n}\\
1 & 1 & a_2 & \cdots & b_{2,n-2} & b_{2,n-1} & b_{2,n}\\
\vdots & \vdots & \vdots & \ddots & \vdots & \vdots & \vdots\\
1 & 1 & 1 & \cdots & a_{n-2} & b_{n-2,n-1} & b_{n-2,n}\\
1 & 1 & 1 & \cdots & 1 & a_{n-1} & b_{n-1,n}\\
1 & 1 & 1 & \cdots & 1 & 1 & a_n\\
\end{pmatrix},
$$
where $a_i\in A$, $b_{i,j}\in B$, and $1\in\mathbbm{k}$, and consequently can then determine the coboundary maps $\partial^n:\overline{C}^n(A,B,\varepsilon)\longrightarrow \overline{C}^{n+1}(A,B,\varepsilon)$ by
\begin{align*}
\partial^nf&\Big(\otimes
\begin{pmatrix}
a_0 & b_{0,1} & \cdots & b_{0,n} & b_{0,n+1}\\
1 & a_1 & \cdots & b_{1,n} & b_{1,n+1}\\
\vdots & \vdots & \ddots & \vdots & \vdots\\
1 & 1 & \cdots & a_{n} & b_{n,n+1}\\
1 & 1 & \cdots & 1 & a_{n+1}\\
\end{pmatrix}\Big)\\
&=\sum_{i=0}^{n}(-1)^if\Big(\otimes
\begin{pmatrix}
a_0 & b_{0,1} & \cdots & b_{0,i}b_{0,i+1} & \cdots & b_{0,n} & b_{0,n+1}\\
1 & a_1 & \cdots & b_{1,i}b_{1,i+1} & \cdots & b_{1,n} & b_{1,n+1}\\
\vdots & \vdots & \ddots & \vdots & \ddots & \vdots & \vdots\\
1 & 1 & \cdots & a_i\varepsilon(b_{i,i+1})a_{i+1} & \cdots & b_{i,n}b_{i+1,n} & b_{i,n+1}b_{i+1,n+1}\\
\vdots & \vdots & \ddots & \vdots & \ddots & \vdots & \vdots\\
1 & 1 & \cdots & 1 & \cdots & a_{n} & b_{n,n+1}\\
1 & 1 & \cdots & 1 & \cdots & 1 & a_{n+1}\\
\end{pmatrix}\Big)\\
&~~~+(-1)^{n+1}f\Big(\otimes
\begin{pmatrix}
a_{n+1}\varepsilon(b_{0,n+1})a_0 & b_{1,n+1}b_{0,1} & \cdots & b_{n-1,n+1}b_{0,n-1} & b_{n,n+1}b_{0,n}\\
1 & a_1 & \cdots & b_{1,n-1} & b_{1,n}\\
\vdots & \vdots & \ddots & \vdots & \vdots\\
1 & 1 & \cdots & a_{n-1} & b_{n-1,n}\\
1 & 1 & \cdots & 1 & a_{n}\\
\end{pmatrix}\Big)
\end{align*}
for all $n\geq0$. It was shown in \cite{LSS} that $\partial^{n+1}\circ\partial^n=0$, and we denote the induced complex by $\overline{\mathbf{C}}^\bullet(A,B,\varepsilon)$.

\begin{definition}(\cite{LSS})
The cohomology of the chain complex $\overline{\mathbf{C}}^\bullet(A,B,\varepsilon)$ is called the \textbf{secondary Hochschild cohomology associated to the triple $(A,B,\varepsilon)$}, and this is denoted by $HH^*(A,B,\varepsilon)$.
\end{definition}

\begin{remark}
By taking $B=\mathbbm{k}$, one can easily see how the secondary Hochschild cohomology associated to the triple $(A,B,\varepsilon)$ reduces to the usual Hochschild cohomology associated to $A$. In notation, we have that $HH^n(A,\mathbbm{k},\varepsilon)=HH^n(A)$ for all $n\geq0$.
\end{remark}

Most meaningfully, we will focus on the chain complex $\overline{\mathbf{C}}^\bullet(A,B,\varepsilon)$ in low dimension. Specifically, we have
$$
0\longrightarrow\Hom_\mathbbm{k}(A,\mathbbm{k})\xrightarrow{~\partial^0~}\Hom_\mathbbm{k}(A^{\otimes2}\otimes B,\mathbbm{k})\xrightarrow{~\partial^1~}\Hom_\mathbbm{k}(A^{\otimes3}\otimes B^{\otimes3},\mathbbm{k})\xrightarrow{~\partial^2~}\ldots
$$
such that
$$
\partial^0f\Big(\otimes\begin{pmatrix} a & \alpha\\ 1 & b\\ \end{pmatrix}\Big)=f(a\varepsilon(\alpha)b)-f(b\varepsilon(\alpha)a)
$$
and
$$
\partial^1f\Big(\otimes\begin{pmatrix} a & \alpha & \beta\\ 1 & b & \gamma\\ 1 & 1 & c\\ \end{pmatrix}\Big)=f\Big(\otimes\begin{pmatrix} a\varepsilon(\alpha)b & \beta\gamma\\ 1 & c\\ \end{pmatrix}\Big)-f\Big(\otimes\begin{pmatrix} a & \alpha\beta\\ 1 & b\varepsilon(\gamma)c\\ \end{pmatrix}\Big)+f\Big(\otimes\begin{pmatrix} c\varepsilon(\beta)a & \alpha\gamma\\ 1 & b\\ \end{pmatrix}\Big).
$$

\subsection{The secondary cyclic cohomology}\label{seccyc}

Connes introduced cyclic cohomology in \cite{C}, and one can see \cite{L} for more details. Here we recall the analogous secondary version that was introduced in \cite{LSS}. We start by considering the permutation $\lambda=(0, 1, 2,\ldots, n)$ and the cyclic group $C_{n+1}=\langle\lambda\rangle$. Notice that $C_{n+1}$ has a natural action on $\overline{C}^n(A,B,\varepsilon)$ given by
$$
\lambda f\Big(\otimes
\begin{pmatrix}
a_0 & b_{0,1} & \cdots & b_{0,n-1} & b_{0,n}\\
1 & a_1 & \cdots & b_{1,n-1} & b_{1,n}\\
\vdots & \vdots & \ddots & \vdots & \vdots\\
1 & 1 & \cdots & a_{n-1} & b_{n-1,n}\\
1 & 1 & \cdots & 1 & a_n\\
\end{pmatrix}\Big)
=(-1)^nf\Big(\otimes
\begin{pmatrix}
a_n & b_{0,n} & \cdots & b_{n-2,n} & b_{n-1,n}\\
1 & a_0 & \cdots & b_{0,n-2} & b_{0,n-1}\\
\vdots & \vdots & \ddots & \vdots & \vdots\\
1 & 1 & \cdots & a_{n-2} & b_{n-2,n-1}\\
1 & 1 & \cdots & 1 & a_{n-1}\\
\end{pmatrix}\Big).
$$

We can then consider the new complex built by setting $\overline{C}_{\lambda}^n(A,B,\varepsilon)=\Ker(1-\lambda)$, where we continue to employ the maps $\partial^n$ from Section \ref{shcsec}. We then get the following definition.

\begin{definition}(\cite{LSS})
The cohomology of the chain complex $\overline{\mathbf{C}}_\lambda^\bullet(A,B\varepsilon)$ is called the \textbf{secondary cyclic cohomology associated to the triple $(A,B,\varepsilon)$}, and this is denoted by $HC^*(A,B,\varepsilon)$.
\end{definition}

As is predictable, we now recall Connes' long exact sequence for the secondary case.

\begin{theorem}\label{SCLESCo}\emph{(\cite{LSS})}
Let $\mathbbm{k}$ be a field of characteristic zero. For a triple $(A,B,\varepsilon)$, we have the long exact sequence
$$\ldots\xrightarrow{I^*}HH^n(A,B,\varepsilon)\xrightarrow{B^*}HC^{n-1}(A,B,\varepsilon)\xrightarrow{S^*}HC^{n+1}(A,B,\varepsilon)\xrightarrow{I^*}HH^{n+1}(A,B,\varepsilon)\xrightarrow{B^*}\ldots$$
\end{theorem}

\subsection{Secondary K\"ahler differentials}

This subsection follows \cite{L2}, where in that paper we saw how the secondary Hochschild homology associated to a commutative triple corresponded to a generalization of K\"ahler differentials.

\begin{definition}(\cite{L2})\label{skdiff}
For a commutative triple $\mathcal{T}=(A,B,\varepsilon)$, denote $\Omega_{\mathcal{T}|\mathbbm{k}}^1$ to be the left $B\otimes A$-module of \textbf{secondary K\"ahler differentials} generated by the $\mathbbm{k}$-linear symbols $d(\alpha\otimes a)$ for $\alpha\in B$ and $a\in A$ with the module structure of $(\alpha\otimes a)\cdot d(\beta\otimes b)=a\varepsilon(\alpha)d(\beta\otimes b)$, along with the relations:
\begin{enumerate}[(i)]
\item\label{og1} $d(\lambda(\alpha\otimes a)+\mu(\beta\otimes b))=\lambda d(\alpha\otimes a)+\mu d(\beta\otimes b)$,
\item\label{og2} $d((\alpha\otimes a)(\beta\otimes b))=a\varepsilon(\alpha)d(\beta\otimes b)+b\varepsilon(\beta)d(\alpha\otimes a)$, and
\item\label{newcon} $d(\alpha\otimes1)+d(\alpha\otimes1)=d(1\otimes\varepsilon(\alpha))$
\end{enumerate}
for all $a,b\in A$, $\alpha,\beta\in B$, and $\lambda,\mu\in\mathbbm{k}$.
\end{definition}

One of the goals of this paper is to showcase how the secondary K\"ahler differentials are indeed nontrivial by way of derivations.

\subsection{The universal derivation}\label{ssecuni}

Finally we recall some facts related to derivations that can be found in such foundational texts like \cite{L} or \cite{W}. These results are classical and commonly recounted as folklore. Furthermore, the goal of Section \ref{sec4} will be to get similar results as these, but for the secondary case.

\begin{definition}\label{universal}
The derivation $D:A\longrightarrow M$ is said to be \textbf{universal} if for any other derivation $\delta:A\longrightarrow N$ there exists a unique $A$-linear map $\varphi:M\longrightarrow N$ such that $\delta=\varphi\circ D$. In other words, the following diagram commutes:
$$
\begin{tikzpicture}[scale=3]
\node (a) at (0,.5) {$A$};
\node (b) at (1,0) {$N$};
\node (c) at (1,1) {$M$};
\path[->,font=\small,>=angle 90]
(a) edge node [above] {$\delta$} (b)
(a) edge node [above] {$D$} (c);
\path[->,dashed,font=\small,>=angle 90]
(c) edge node [right] {$\exists!\varphi$} (b);
\end{tikzpicture}
$$
\end{definition}

\begin{proposition}\label{UniProp}
For $A$ commutative, we have that the map $d:A\longrightarrow\Omega_{A|\mathbbm
{k}}^1$ given by
$$a\longmapsto d(a)$$
is the universal derivation.
\end{proposition}

\begin{proposition}\label{DerDiffHomo}
For $A$ commutative, we have that
$$\Hom_A(\Omega_{A|\mathbbm{k}}^1,M)\longrightarrow\Der_\mathbbm{k}(A,M)$$
given by $f\longmapsto f\circ d$ is an isomorphism.
\end{proposition}

\section{Secondary Derivations}\label{sec3}

In the usual case, one has the classic result of $HH^1(A)\cong\Der_\mathbbm{k}(A,A^*)$, where $HH^1(A)$ denotes the Hochschild cohomology of $A$ in dimension $1$, and $\Der_\mathbbm{k}(A,A^*)$ denotes the set of all $\mathbbm{k}$-linear derivations from $A$ to $A^*$. Furthermore, one can also conclude that $HC^1(A)\cong\Der_\mathbbm{k}^1(A,A^*)$, where the superscript $1$ means we add the condition $D(a\otimes b)=-D(b\otimes a)$ to the aforementioned set $\Der_\mathbbm{k}(A,A^*)$.

The main goal of this section is to get results in the secondary case that corresponds to the above.

\begin{definition}\label{SDeriv}
A \textbf{secondary derivation} of the commutative triple $\mathcal{T}=(A,B,\varepsilon)$ with values in $M$ is a $\mathbbm{k}$-linear map $D:B\otimes A\longrightarrow M$ with the symmetric bimodule structure of $(\alpha\otimes a)\cdot D(\beta\otimes b)=a\varepsilon(\alpha)D(\beta\otimes b)=D(\beta\otimes b)a\varepsilon(\alpha)=D(\beta\otimes b)\cdot(\alpha\otimes a)$ such that
\begin{enumerate}[(i)]
    \item $D(\lambda(\alpha\otimes a)+\mu(\beta\otimes b))=\lambda D(\alpha\otimes a)+\mu D(\beta\otimes b)$,
    \item $D((\alpha\otimes a)(\beta\otimes b))=a\varepsilon(\alpha)D(\beta\otimes b)+D(\alpha\otimes a)b\varepsilon(\beta)$, and
    \item $D(\alpha\otimes1)+D(\alpha\otimes1)=D(1\otimes\varepsilon(\alpha))$
\end{enumerate}
for all $a,b\in A$, $\alpha,\beta\in B$, and $\lambda,\mu\in\mathbbm{k}$. The set of all such secondary derivations is denoted by $\Der_\mathbbm{k}(\mathcal{T},M)$.
\end{definition}

Notice that we have $(\alpha\otimes a)(\beta\otimes b)=\alpha\beta\otimes ab$, and therefore
$$
D(\alpha\beta\otimes ab)=a\varepsilon(\alpha)D(\beta\otimes b)+D(\alpha\otimes a)b\varepsilon(\beta).
$$
As consequence, it is then immediate that
\begin{align*}
D(1\otimes ab)&=aD(1\otimes b)+D(1\otimes a)b,\\
D(\alpha\beta\otimes1)&=\varepsilon(\alpha)D(\beta\otimes1)+D(\alpha\otimes1)\varepsilon(\beta),\text{and}\\
D(\alpha\otimes a)&=\varepsilon(\alpha)D(1\otimes a)+D(\alpha\otimes1)a.
\end{align*}

One could wonder if secondary derivations have a Lie algebra structure. This could be an avenue worth pursuing in future work.

\begin{remark}
It is easy to see that $D(1\otimes1)=0$, and hence $D(\lambda\otimes\mu)=0$ for all $\lambda,\mu\in\mathbbm{k}$. Furthermore, we note that the first two conditions of Definition \ref{SDeriv} make $D$ a derivation of the commutative algebra $B\otimes A$ with a symmetric bimodule structure, while the third condition is additional.
\end{remark}

\begin{remark}
Under the identification of $B=\mathbbm{k}$, notice that secondary derivations become the usual derivations $\Der_\mathbbm{k}(A,M)$ for the commutative $\mathbbm{k}$-algebra $A$ into the $A$-symmetric bimodule $M$. In particular, one can see that the final condition in Definition \ref{SDeriv} becomes trivial when $B=\mathbbm{k}$.
\end{remark}

\begin{example}\label{goodex}
With $A=M=\mathbbm{k}[x]$, $B=\mathbbm{k}[x^2]$, and $\iota:\mathbbm{k}[x^2]\longrightarrow\mathbbm{k}[x]$ given by inclusion, we claim that for the commutative triple $\mathcal{T}=(\mathbbm{k}[x],\mathbbm{k}[x^2],\iota)$, its corresponding set of secondary derivations $\Der_\mathbbm{k}(\mathcal{T},\mathbbm{k}[x])$ is nontrivial. Note $D\in\Der_\mathbbm{k}(\mathcal{T},\mathbbm{k}[x])$ when we define $D(g\otimes1)=g'/2$, $D(1\otimes h)=h'$, and $D(g\otimes h)=gh'+hg'/2$ for $g\in\mathbbm{k}[x^2]$ and $h\in\mathbbm{k}[x]$.
\end{example}

Before we turn our attention to the case when $M=A^*=\Hom_\mathbbm{k}(A,\mathbbm{k})$, there are a couple straightforward computations and an observation that will prove useful.

\begin{example}\label{HH0ex}
For any triple $(A,B,\varepsilon)$, it is easy to see that
$$
HH^0(A,B,\varepsilon)=HH^0(A)=(A^*)^A=\{f:A\longrightarrow\mathbbm{k}~|~f(ab)=f(ba)\text{~for all~}a,b\in A\}.
$$
Furthermore, using Theorem \ref{SCLESCo}, it is immediate that $HC^0(A,B,\varepsilon)\cong HH^0(A,B,\varepsilon)$. Of particular interest is when we have a commutative triple $(A,B,\varepsilon)$, we get that the map $\partial^0\equiv0$. This implies that $\Ig(\partial^0)$ is trivial.
\end{example}

\begin{theorem}\label{HH1ABCo}
For a commutative triple $\mathcal{T}=(A,B,\varepsilon)$, we have that
$$HH^1(A,B,\varepsilon)\cong\Der_\mathbbm{k}(\mathcal{T},A^*).$$
\end{theorem}
\begin{proof}
As stated in Example \ref{HH0ex}, since $A$ is commutative, we get that $\Ig(\partial^0)$ is trivial, and so $HH^1(A,B,\varepsilon)=\Ker(\partial^1)$. Furthermore, $\Ker(\partial^1)$ consists of all $\mathbbm{k}$-linear maps from $A^{\otimes2}\otimes B$ to $\mathbbm{k}$ such that
\begin{equation}\label{HHrel}
f\Big(\otimes\begin{pmatrix} a\varepsilon(\alpha)b & \beta\gamma\\ 1 & c\\ \end{pmatrix}\Big)-f\Big(\otimes\begin{pmatrix} a & \alpha\beta\\ 1 & b\varepsilon(\gamma)c\\ \end{pmatrix}\Big)+f\Big(\otimes\begin{pmatrix} c\varepsilon(\beta)a & \alpha\gamma\\ 1 & b\\ \end{pmatrix}\Big)=0.
\end{equation}

Next, we observe that under the identification $M=A^*$, we have that $\Der_\mathbbm{k}(\mathcal{T},A^*)$ also consists of $\mathbbm{k}$-linear maps from $A^{\otimes2}\otimes B$ to $\mathbbm{k}$, but has the following two relations:
\begin{equation}\label{Derrel}
\begin{gathered}
D\Big(\otimes\begin{pmatrix} a & \alpha\beta\\1 & bc\end{pmatrix}\Big)=D\Big(\otimes\begin{pmatrix} ab\varepsilon(\alpha) & \beta\\ 1 & c\\\end{pmatrix}\Big)+D\Big(\otimes\begin{pmatrix} ca\varepsilon(\beta) & \alpha\\ 1 & b\\\end{pmatrix}\Big)\text{~~~and}\\
D\Big(\otimes\begin{pmatrix} a & \gamma\\ 1 & 1\end{pmatrix}\Big)+D\Big(\otimes\begin{pmatrix} a & \gamma\\ 1 & 1\end{pmatrix}\Big)=D\Big(\otimes\begin{pmatrix} a & 1\\ 1 & \varepsilon(\gamma)\end{pmatrix}\Big).
\end{gathered}
\end{equation}

To get our desired isomorphism, it is sufficient to show \eqref{HHrel} implies \eqref{Derrel}, as well as the converse.

Supposing \eqref{HHrel}, it is easy to see how \eqref{Derrel} is satisfied: simply take $\gamma=1_B$ to get the first equation, while taking $b=c=1_A$ and $\alpha=\beta=1_B$ obtains the second equation.

On the other hand, suppose \eqref{Derrel}.  Notice that we have
\begin{align*}
f\Big(\otimes\begin{pmatrix} a & \alpha\beta\\ 1 & bc\varepsilon(\gamma)\\ \end{pmatrix}\Big)&=f\Big(\otimes\begin{pmatrix} a\varepsilon(\alpha\beta) & 1\\ 1 & bc\varepsilon(\gamma)\\ \end{pmatrix}\Big)+f\Big(\otimes\begin{pmatrix} abc\varepsilon(\gamma) & \alpha\beta\\ 1 & 1\\ \end{pmatrix}\Big)\\
&=f\Big(\otimes\begin{pmatrix} a\varepsilon(\alpha\beta)bc & 1\\ 1 & \varepsilon(\gamma)\\ \end{pmatrix}\Big)+f\Big(\otimes\begin{pmatrix} a\varepsilon(\alpha\beta)b\varepsilon(\gamma) & 1\\ 1 & c\\ \end{pmatrix}\Big)\\
&~~~+f\Big(\otimes\begin{pmatrix} a\varepsilon(\alpha\beta)c\varepsilon(\gamma) & 1\\ 1 & b\\ \end{pmatrix}\Big)+f\Big(\otimes\begin{pmatrix} abc\varepsilon(\gamma)\varepsilon(\alpha) & \beta\\ 1 & 1\\ \end{pmatrix}\Big)\\
&~~~+f\Big(\otimes\begin{pmatrix} abc\varepsilon(\gamma)\varepsilon(\beta) & \alpha\\ 1 & 1\\ \end{pmatrix}\Big)\\
&=f\Big(\otimes\begin{pmatrix} a\varepsilon(\alpha\beta)bc & \gamma\\ 1 & 1\\ \end{pmatrix}\Big)+f\Big(\otimes\begin{pmatrix} a\varepsilon(\alpha\beta)bc & \gamma\\ 1 & 1\\ \end{pmatrix}\Big)\\
&~~~+f\Big(\otimes\begin{pmatrix} a\varepsilon(\alpha\beta)b\varepsilon(\gamma) & 1\\ 1 & c\\ \end{pmatrix}\Big)+f\Big(\otimes\begin{pmatrix} a\varepsilon(\alpha\beta)c\varepsilon(\gamma) & 1\\ 1 & b\\ \end{pmatrix}\Big)\\
&~~~+f\Big(\otimes\begin{pmatrix} abc\varepsilon(\gamma\alpha) & \beta\\ 1 & 1\\ \end{pmatrix}\Big)+f\Big(\otimes\begin{pmatrix} abc\varepsilon(\gamma\beta) & \alpha\\ 1 & 1\\ \end{pmatrix}\Big)\\
&=f\Big(\otimes\begin{pmatrix} ab\varepsilon(\alpha)c & \beta\gamma\\ 1 & 1\\ \end{pmatrix}\Big)+f\Big(\otimes\begin{pmatrix} ab\varepsilon(\alpha)\varepsilon(\beta\gamma) & 1\\ 1 & c\\ \end{pmatrix}\Big)\\
&~~~+f\Big(\otimes\begin{pmatrix} ac\varepsilon(\beta)b & \alpha\gamma\\ 1 & 1\\ \end{pmatrix}\Big)+f\Big(\otimes\begin{pmatrix} ac\varepsilon(\beta)\varepsilon(\alpha\gamma) & 1\\ 1 & b\\ \end{pmatrix}\Big)\\
&=f\Big(\otimes\begin{pmatrix} ab\varepsilon(\alpha) & \beta\gamma\\ 1 & c\\ \end{pmatrix}\Big)+f\Big(\otimes\begin{pmatrix} ca\varepsilon(\beta) & \alpha\gamma\\ 1 & b\\ \end{pmatrix}\Big),
\end{align*}
which was what we wanted. Thus, the isomorphism follows.
\end{proof}

\begin{corollary}\label{HC1ABCo}
For a commutative triple $\mathcal{T}=(A,B,\varepsilon)$, we have that
$$HC^1(A,B,\varepsilon)\cong\Der_\mathbbm{k}^1(\mathcal{T},A^*),$$
where the superscript $1$ denotes we add the condition
$$D\Big(\otimes\begin{pmatrix} a & \alpha\\ 1 & b\\ \end{pmatrix}\Big)=-D\Big(\otimes\begin{pmatrix} b & \alpha\\ 1 & a\\ \end{pmatrix}\Big)$$
to $\Der_\mathbbm{k}(\mathcal{T},A^*)$.
\end{corollary}
\begin{proof}
First we observe that $\mathcal{T}=(A,B,\varepsilon)$ is a commutative triple, and so Proposition \ref{HH1ABCo} will apply when we use Theorem \ref{SCLESCo} in low dimension. Specifically, we have that
$$
0\longrightarrow HC^1(A,B,\varepsilon)\xrightarrow{~I^*~}\Der_\mathbbm{k}(\mathcal{T},A^*)\xrightarrow{~B^*~}\ldots
$$
Thus, by exactness and the first isomorphism theorem, one gets
$$
HC^1(A,B,\varepsilon)\cong\frac{HC^1(A,B,\varepsilon)}{\{0\}}=\frac{HC^1(A,B,\varepsilon)}{\Ker(I^*)}\cong\Ig(I^*)\subseteq\Der_\mathbbm{k}(\mathcal{T},A^*).
$$
Therefore, since $I^*$ is induced by inclusion, we have that $\Ig(I^*)$ contains the cyclic maps described in Section \ref{seccyc}, which are those that satisfy
$$
D\Big(\otimes\begin{pmatrix} a & \alpha\\ 1 & b\\ \end{pmatrix}\Big)=-D\Big(\otimes\begin{pmatrix} b & \alpha\\ 1 & a\\ \end{pmatrix}\Big),
$$
as desired.
\end{proof}

\begin{remark}
Notice that Theorem \ref{HH1ABCo} and Corollary \ref{HC1ABCo} reduce to the expected classical results when we take $B=\mathbbm{k}$, as described at the start of Section \ref{sec3}.
\end{remark}

\section{The Universal Property}\label{sec4}

The purpose of this section is twofold: to highlight a universal property for the secondary derivations (introduced in Definition \ref{SDeriv}), and to show that the secondary K\"ahler differentials from \cite{L2} are nontrivial (see Definition \ref{skdiff}). In particular, the results from this section will run parallel to what was recalled in Section \ref{ssecuni}.

\begin{definition}
The secondary derivation $D:B\otimes A\longrightarrow M$ is said to be \textbf{universal} if for any other secondary derivation $\delta:B\otimes A\longrightarrow N$ there exists a unique $A$-linear map $\varphi:M\longrightarrow N$ such that $\delta=\varphi\circ D$. In other words, the following diagram commutes:
$$
\begin{tikzpicture}[scale=3]
\node (a) at (0,.5) {$B\otimes A$};
\node (b) at (1,0) {$N$};
\node (c) at (1,1) {$M$};
\path[->,font=\small,>=angle 90]
(a) edge node [above] {$\delta$} (b)
(a) edge node [above] {$D$} (c);
\path[->,dashed,font=\small,>=angle 90]
(c) edge node [right] {$\exists!\varphi$} (b);
\end{tikzpicture}
$$
\end{definition}

\begin{proposition}\label{SDerDiff}
For a commutative triple $\mathcal{T}=(A,B,\varepsilon)$, we have that the map
$$d:B\otimes A\longrightarrow\Omega_{\mathcal{T}|\mathbbm{k}}^1$$
given by $\alpha\otimes a\longmapsto d(\alpha\otimes a)$ is the universal secondary derivation.
\end{proposition}
\begin{proof}
In order to verify this map is a secondary derivation, there are three conditions to check from Definition \ref{SDeriv}. However, these all follow immediately from the definition of secondary K\"ahler differentials $\Omega_{\mathcal{T}|\mathbbm{k}}^1$ (see Definition \ref{skdiff}), and because $A$ is commutative. Thus $d:B\otimes A\longrightarrow\Omega_{\mathcal{T}|\mathbbm{k}}^1$ is a secondary derivation.

The map is also universal since for any derivation $\delta:B\otimes A\longrightarrow N $ there is a unique map $\varphi:\Omega_{\mathcal{T}|\mathbbm{k}}^1\longrightarrow N$ determined by $d(\alpha\otimes a)\longmapsto\delta(\alpha\otimes a)$. It is then immediate that $\delta=\varphi\circ d$ by construction.
\end{proof}

\begin{proposition}\label{SDerDiffHomo}
For a commutative triple $\mathcal{T}=(A,B,\varepsilon)$, we have that
$$\Hom_A(\Omega_{\mathcal{T}|\mathbbm{k}}^1,M)\longrightarrow\Der_\mathbbm{k}(\mathcal{T},M)$$
given by $f\longmapsto f\circ d$ is an isomorphism.
\end{proposition}
\begin{proof}
We first note that the domain consists of $A$-linear morphisms, which will play a role below. Next we show that $f\circ d$ is a secondary derivation. From Definition \ref{SDeriv}, there are three conditions to check; condition (i) is clear, but for (ii) we have
\begin{align*}
(f\circ d)((\alpha\otimes a)(\beta\otimes b))&=f(d((\alpha\otimes a)(\beta\otimes b)))\\
&=f(a\varepsilon(\alpha)d(\beta\otimes b)+b\varepsilon(\beta)d(\alpha\otimes a))\\
&=f(a\varepsilon(\alpha)d(\beta\otimes b))+f(b\varepsilon(\beta)d(\alpha\otimes a))\\
&=a\varepsilon(\alpha)f(d(\beta\otimes b))+b\varepsilon(\beta)f(d(\alpha\otimes a))\\
&=a\varepsilon(\alpha)(f\circ d)(\beta\otimes b)+b\varepsilon(\beta)(f\circ d)(\alpha\otimes a),
\end{align*}
and for (iii) we have
\begin{align*}
(f\circ d)(\alpha\otimes1)+(f\circ d)(\alpha\otimes1)&=f(d(\alpha\otimes1))+f(d(\alpha\otimes1))\\
&=f(d(\alpha\otimes1)+d(\alpha\otimes1))\\
&=f(d(1\otimes\varepsilon(\alpha)))\\
&=(f\circ d)(1\otimes\varepsilon(\alpha)).
\end{align*}
Thus, $f\circ d$ is a secondary derivation. The isomorphism follows from the universality of $\Omega_{\mathcal{T}|\mathbbm{k}}^1$, coming from Proposition \ref{SDerDiff}.
\end{proof}

\begin{remark}
By taking $B=\mathbbm{k}$, note that these results will reduce to the usual case, as described in Section \ref{ssecuni}.
\end{remark}

\begin{remark}
Due to the isomorphism in Proposition \ref{SDerDiffHomo} and by Example \ref{goodex} (for instance), we conclude that the secondary K\"ahler differentials $\Omega_{\mathcal{T}|\mathbbm{k}}^1$ are nontrivial.
\end{remark}


\end{document}